\documentclass[a4paper,12pt]{article}

\usepackage{amsmath}
\usepackage{amssymb}
\usepackage{amsthm}

\newtheorem{theorem}{Theorem}[section]

\newtheorem{corollary}{Corollary}[section]

\begin{document}

\title{Arithmetic Subderivatives and Leibniz-Additive Functions}
\author{Jorma K. Merikoski, Pentti Haukkanen\\
Faculty of Information Technology and
Communication Sciences,\\ FI-33014 Tampere University,
 Finland\\
{jorma.merikoski@uta.fi}, {pentti.haukkanen@tuni.fi}\\
Timo Tossavainen\\
Department of Arts, Communication and Education,\\ Lulea University of Technology, \\
 SE-97187 Lulea, Sweden\\
{timo.tossavainen@ltu.se}
}

\date{}
\maketitle

\begin{abstract}
We first introduce the arithmetic subderivative of a positive integer with respect to a non-empty
set of primes. This notion generalizes the concepts of the arithmetic derivative and arithmetic
partial derivative. More generally, we then define that an arithmetic function~$f$ is Leibniz-additive
if there is a nonzero-valued and completely multiplicative function~$h_f$ satisfying
$f(mn)=f(m)h_f(n)+f(n)h_f(m)$ for all positive integers $m$ and~$n$.
We study some basic properties of such functions. For example, we present conditions when an
arithmetic function is Leibniz-additive and, generalizing well-known bounds
for the arithmetic derivative, establish bounds for a Leibniz-additive function.
%\keywords arithmetic derivative, Leibniz rule, additivity, multiplicativity
%\AMSclassificationnumber 11A25, 11A05
\end{abstract}

\section{Introduction}

We let $\mathbb{P}$, $\mathbb{Z}_+$, $\mathbb{N}$, $\mathbb{Z}$, and
$\mathbb{Q}$ stand
for the set of primes,  positive integers, nonnegative integers, integers, and
rational numbers, respectively.

Let $n\in\mathbb{Z}_+$. There is a unique sequence
$(\nu_p(n))_{p\in\mathbb{P}}$ of nonnegative integers (with only finitely many
positive terms) such that
\begin{eqnarray}
\label{n}
n=\prod_{p\in\mathbb{P}}p^{\nu_p(n)}.
\end{eqnarray}
We use this notation throughout.

Let $\emptyset\ne S\subseteq\mathbb{P}$. We define the {\em arithmetic
subderivative} of~$n$ with respect to~$S$ as
$$
D_S(n)=n'_S=n\sum_{p\in S}\frac{\nu_p(n)}{p}.
$$
In particular, $n'_\mathbb{P}$ is the {\em arithmetic derivative} of~$n$, defined by
Barbeau~\cite{Ba} and studied further by Ufnarovski and \AA hlander~\cite{UA}.
Another well-known special case is $n'_{\{p\}}$, the {\em arithmetic partial
derivative} of~$n$ with respect to $p\in\mathbb{P}$, defined by Kovi\v{c}~\cite{Ko}
and studied further by the present authors and Mattila~\cite{HMMT, HMT1}.

We define the {\em arithmetic logarithmic subderivative} of~$n$ with respect to~$S$ as
$$
{\rm ld}_S(n)=\frac{D_S(n)}{n}=\sum_{p\in S}\frac{\nu_p(n)}{p}.
$$
In particular, ${\rm ld}_\mathbb{P}(n)$ is the {\em arithmetic logarithmic
derivative} of~$n$. This notion was originally introduced by Ufnarovski and \AA hlander \cite{UA}.

An arithmetic function~$g$ is
{\it completely additive} (or {\em c-additive}, for short) if $g(mn)=g(m)+g(n)$ for all
$m,n\in\mathbb{Z}_+$. 
It follows from the definition that $g(1)=0$. An arithmetic function~$h$ is
{\it completely multiplicative} (or {\em c-multiplicative}, for short) if $h(1)=1$ and $h(mn)=h(m)h(n)$ for all
$m,n\in\mathbb{Z}_+$. The following theorems recall
that these functions are totally determined by their values at primes.
The proofs are simple and omitted.

\begin{theorem}
\label{cadd}
Let $g$ be an arithmetic function, and let
$(x_p)_{p\in\mathbb{P}}$ be a sequence of real numbers. The following
conditions are equivalent:
\begin{enumerate}
\item[{\rm (a)}] $g$ is c-additive and $g(p)=x_p$ for all $p\in\mathbb{P}$;
\item[{\rm (b)}] for all $n\in\mathbb{Z}_+$,
$$
g(n)=\sum_{p\in\mathbb{P}}\nu_p(n)x_p.
$$
\end{enumerate}
\end{theorem}

\begin{theorem}
\label{cmult}
Let  $h$ be an arithmetic and nonzero-valued function, and let
$(y_p)_{p\in\mathbb{P}}$ be a sequence of nonzero real numbers. The following
conditions are equivalent:
\begin{enumerate}
\item[{\rm (a)}] $h$ is c-multiplicative and $h(p)=y_p$ for all $p\in\mathbb{P}$;
\item[{\rm (b)}] for all $n\in\mathbb{Z}_+$,
$$
h(n)=\prod_{p\in\mathbb{P}}y_p^{\nu_p(n)}.
$$
\end{enumerate}
\end{theorem}

We say that an arithmetic function $f$ is {\em Leibniz-additive}
(or {\em L-additive}, for short) if there is a nonzero-valued and c-multiplicative function $h_f$ such that 
\begin{equation}\label{leibniz}
f(mn)=f(m)h_f(n)+f(n)h_f(m)
\end{equation}
for all $m,n\in\mathbb{Z}_+$. 
Then $f(1)=0$, since $h_f(1)=1$. The property~\eqref{leibniz} may be considered
a generalized Leibniz rule. Substituting $m=n=p\in\mathbb{P}$ and
applying induction, we get
\begin{eqnarray}
\label{fpa}
f(p^a)=af(p)h(p)^{a-1}
\end{eqnarray}
for all $p\in\mathbb{P}$, $a\in\mathbb{Z}_+$.
 
The arithmetic subderivative~$D_S$ is L-additive with $h_{D_S}=N$, where $N$ is the
identity function $N(n)=n$. A c-additive
function~$g$ is L-additive with $h_g=E$, where $E(n)=1$ for all $n\in\mathbb{Z}_+$.
The arithmetic logarithmic subderivative~${\rm ld}_S$ is c-additive and hence L-additive.

This paper is a sequel to~\cite{HMT2}, where we defined L-additivity without
requiring that $h_f$ is nonzero-valued.
We begin by showing how the values of an L-additive function $f$ are determined in $\mathbb{Z}_+$ by the values of $f$ and $h_f$ at primes (Section 2) and then study under which conditions an arithmetic function~$f$ can be expressed as $f=gh$, where $g$ is c-additive and
$h$ is nonzero-valued and c-multiplicative (Section~3). It turns out that the same
conditions are necessary for L-additivity (Section 4). Finally, extending
Barbeau's~\cite{Ba} and Westrick's~\cite{We} results, we present some lower and upper bounds for an L-additive function (Section 5). We complete our paper with some remarks (Section 6).

\section{Constructing $f(n)$ and $h_f(n)$}
\label{Constructing}

An L-additive function~$f$ is not totally defined by its values at primes. Also,
the values of~$h_f$ at primes must be known.

\begin{theorem}
\label{ladd}
Let $f$ be an arithmetic function, and let
$(x_p)_{p\in\mathbb{P}}$ and $(y_p)_{p\in\mathbb{P}}$ be as in
Theorems~{\rm \ref{cadd}} and~{\rm \ref{cmult}}. The following conditions are equivalent:
\begin{enumerate}
\item[{\rm (a)}] $f$ is L-additive and
$f(p)=x_p$, $h_f(p)=y_p$ for all $p\in\mathbb{P}$;
\item[{\rm (b)}] for all $n\in\mathbb{Z}_+$,
$$
f(n)=\Big(\sum_{p\in\mathbb{P}}\nu_p(n)\frac{x_p}{y_p}\Big)
\prod_{p\in\mathbb{P}}
y_p^{\nu_p(n)}.
$$
\end{enumerate}
\end{theorem}

\begin{proof}
(a)$\Rightarrow$(b). Since $f(1)=0$, (b) holds for $n=1$. So, let $n>1$. Denoting
$$
\{p_1,\dots,p_s\}=\{p\in\mathbb{P}\mid\nu_p(n)>0\}
$$
and
$$
a_i=\nu_{p_i}(n),\quad i=1,\dots,s,
$$
we have
\begin{eqnarray*}
f(n)=\sum_{i=1}^sh_f(p_1)^{a_1}\cdots h_f(p_{i-1})^{a_{i-1}}f(p_i^{a_i})
h_f(p_{i+1})^{a_{i+1}}\cdots h_f(p_r)^{a_r}=\qquad
\\
\sum_{i=1}^sh_f(p_1)^{a_1}\cdots h_f(p_{i-1})^{a_{i-1}}a_if(p_i)h_f(p_i)^{a_i-1}
h_f(p_{i+1})^{a_{i+1}}\cdots h_f(p_r)^{a_r}=\qquad
\\
\sum_{p\in\mathbb{P}}\Big(\nu_p(n)f(p)h_f(p)^{\nu_p(n)-1}
\prod_{\substack{q\in\mathbb{P}\\q\ne p}}h_f(q)^{\nu_q(n)}\Big)=
\sum_{p\in\mathbb{P}}\Big(\nu_p(n)\frac{f(p)}{h_f(p)}
\prod_{q\in\mathbb{P}}h_f(q)^{\nu_q(n)}\Big)=
\\
\Big(\sum_{p\in\mathbb{P}}\nu_p(n)\frac{x_p}{y_p}\Big)
\prod_{p\in\mathbb{P}}y_p^{\nu_p(n)}.
\end{eqnarray*}
The first equation can be proved by induction on~$r$, the second holds
by~(\ref{fpa}), and the remaining equations are obvious.

(b)$\Rightarrow$(a). We define now
$$
h(n)=\prod_{p\in\mathbb{P}}y_p^{\nu_p(n)}.
$$
Let $m,n\in\mathbb{Z}_+$. Then
\begin{eqnarray*}
f(mn)=\Big(\sum_{p\in\mathbb{P}}\nu_p(mn)\frac{x_p}{y_p}\Big)
\prod_{p\in\mathbb{P}}y_p^{\nu_p(mn)}=
\\
\Big(\sum_{p\in\mathbb{P}}(\nu_p(m)+\nu_p(n))\frac{x_p}{y_p}\Big)
\prod_{p\in\mathbb{P}}y_p^{\nu_p(m)+\nu_p(n)}=
\\
\Big(\sum_{p\in\mathbb{P}}(\nu_p(m)+\nu_p(n))\frac{x_p}{y_p}\Big)
\Big(\prod_{p\in\mathbb{P}}y_p^{\nu_p(m)}\Big)
\Big(\prod_{p\in\mathbb{P}}y_p^{\nu_p(n)}\Big)=
\\
\Big(\sum_{p\in\mathbb{P}}\nu_p(m)\frac{x_p}{y_p}
\Big(\prod_{p\in\mathbb{P}}y_p^{\nu_p(m)}\Big)\Big)
\Big(\prod_{p\in\mathbb{P}}y_p^{\nu_p(n)}\Big)+
\\
\Big(\sum_{p\in\mathbb{P}}\nu_p(n)\frac{x_p}{y_p}
\Big(\prod_{p\in\mathbb{P}}y_p^{\nu_p(n)}\Big)\Big)
\Big(\prod_{p\in\mathbb{P}}y_p^{\nu_p(m)}\Big)=
\\
f(m)h(n)+f(n)h(m).
\end{eqnarray*}
So, $f$ is L-additive with $h_f=h$. It is clear that $f(p)=x_p$ and $h_f(p)=y_p$
for all $p\in\mathbb{P}$.
\end{proof}

Next, we construct~$h_f$ from~$f$. Let us denote
$$
U_f=\{p\in\mathbb{P}\mid f(p)\ne 0\},\quad V_f=\{p\in\mathbb{P}\mid f(p)=0\}.
$$
If $f=\theta$, where $\theta(n)=0$ for all $n\in\mathbb{Z}_+$, then
any $h_f$ applies. Hence, we now assume
that $f\ne\theta$. Then $U_f\ne\emptyset$.

Since
$$
f(p^2)=2f(p)h_f(p)
$$
by~(\ref{fpa}), we have
$$
h_f(p)=\frac{f(p^2)}{2f(p)}\quad{\rm for}\,\,p\in U_f.
$$
The case $p\in V_f$ remains. Let $q\in\mathbb{P}$. Then (\ref{leibniz}) implies that
$$
f(pq)=f(p)h_f(q)+f(q)h_f(p)=f(q)h_f(p).
$$
Therefore,
\begin{eqnarray}
\label{hfpq}
h_f(p)=\frac{f(pq)}{f(q)}\quad{\rm for}\,\,p\in V_f,
\end{eqnarray}
where $q\in U_f$ is arbitrary. Now, by Theorem~\ref{cmult},
\begin{eqnarray}
\label{hf}
h_f(n)=\Big(\prod_{p\in U_f}\Big(\frac{f(p^2)}{2f(p)}\Big)^{\nu_p(n)}\Big)
\Big(\prod_{p\in V_f}\Big(\frac{f(pq)}{f(q)}\Big)^{\nu_p(n)}\Big),
\end{eqnarray}
where $q\in U_f$ is arbitrary. (If $V_f=\emptyset$, then the latter factor is
the ``empty product'' one.)
We have thus proved the following theorem.

\begin{theorem}
\label{uniqhf}
If $f\ne\theta$ is L-additive, then
$h_f$ is unique and determined by~{\rm (\ref{hf})}.
\end{theorem}

\section{Decomposing $f=gh$}
\label{Decomposing}

Let $f$ be an arithmetic function and let $h$ be a nonzero-valued and c-multiplicative
function. By Theorem~\ref{ladd}, $f$ is L-additive with $h_f=h$ if and only if
\begin{eqnarray}
\label{f}
f(n)=\Big(\sum_{p\in\mathbb{P}}\nu_p(n)\frac{f(p)}{h(p)}\Big)
\prod_{p\in\mathbb{P}}h(p)^{\nu_p(n)}=
\Big(\sum_{p\in\mathbb{P}}\nu_p(n)\frac{f(p)}{h(p)}\Big)h(n).
\end{eqnarray}
The function
$$
g(n)=\sum_{p\in\mathbb{P}}\nu_p(n)\frac{f(p)}{h(p)}
$$
is c-additive by Theorem~\ref{cadd}.

We say that an arithmetic function~$f$ is {\em gh-decomposable} if it has a {\em gh decomposition}
$$
f=gh,
$$
where $g$ is c-additive and $h$ is nonzero-valued and c-multiplicative. We saw above
that L-additivity implies $gh$-decomposability. Also, the converse holds.

\begin{theorem}
\label{fgh}
Let $f$ be an arithmetic function. The following conditions are equivalent:
\begin{enumerate}
\item[{\rm (a)}] $f$ is L-additive;
\item[{\rm (b)}] $f$ is $gh$-decomposable.
\end{enumerate}
\end{theorem}

\begin{proof}
(a)$\Rightarrow$(b). We proved this above.

(b)$\Rightarrow$(a). For all $m,n\in\mathbb{Z}_+$,
\begin{eqnarray*}
f(mn)=g(mn)h(mn)=(g(m)+g(n))h(m)h(n)=\qquad
\\
g(m)h(m)h(n)+g(n)h(n)h(m)=f(m)h(n)+f(n)h(m).
\end{eqnarray*}
Consequently, $f$ is L-additive with $h_f=h$.
\end{proof}

\begin{corollary}
\label{fghuniq}
Let $f\ne\theta$ be an arithmetic function. The following conditions are equivalent:
\begin{enumerate}
\item[{\rm (a)}] $f$ is L-additive;
\item[{\rm (b)}] $f$ is uniquely $gh$-decomposable.
\end{enumerate}
\end{corollary}

\begin{proof}
In proving (a)$\Rightarrow$(b), $h_f$ is unique by Theorem~\ref{uniqhf}. Since $h_f$ is
nonzero-valued, also $g=f/h_f$ is unique.
\end{proof}

For example, if $f=D_S$, then $g={\rm ld}_S$ and $h=N$.

By Theorem~\ref{uniqhf}, an L-additive function $f\ne\theta$ determines~$h_f$ uniquely. We consider next the converse problem: Given a nonzero-valued and c-multiplicative function~$h$, find an L-additive function~$f$
such that $h_f=h$.

\begin{theorem}
\label{ladduniq}
Let $(x_p)_{p\in\mathbb{P}}$ be a sequence of real numbers and let $h$
be nonzero-valued and c-multiplicative. There is a unique L-additive
function~$f$
with $h_f=h$ such that $f(p)=x_p$ for all $p\in\mathbb{P}$.
\end{theorem}

\begin{proof}
If at least one $x_p\ne 0$, then apply Theorem~\ref{ladd} and
Corollary~\ref{fghuniq}. Otherwise, $f=\theta$.
\end{proof}

We can now characterize $D_S$ and ${\rm ld}_S$.

\begin{corollary}
\label{dschar}
Let $f$ be an arithmetic function and $\emptyset\ne S\subseteq\mathbb{P}$. The following
conditions are equivalent:
\begin{enumerate}
\item[{\rm (a)}] $f$ is L-additive, $h_f=N$, $f(p)=1$ for $p\in S$, and $f(p)=0$
for $p\in\mathbb{P}\setminus S$;
\item[{\rm (b)}] $f=D_S$.
\end{enumerate}
\end{corollary}

\begin{corollary}
\label{ldschar}
Let $g$ be an arithmetic function and $\emptyset\ne S\subseteq\mathbb{P}$. The following
conditions are equivalent:
\begin{enumerate}
\item[{\rm (a)}] $g$ is c-additive, $g(p)=1/p$ for $p\in S$, and $g(p)=0$
for $p\in\mathbb{P}\setminus S$;
\item[{\rm (b)}] $g={\rm ld}_S$.
\end{enumerate}
\end{corollary}

\section{Conditions for L-additivity}
\label{Conditions}

Let $f\ne\theta$ be L-additive and $a,b\in\mathbb{N}$.

First, let $p\in\mathbb{P}$. By~(\ref{fpa}),
\begin{eqnarray}
\label{fpeqs1}
f(p^{a+1})=(a+1)f(p)h_f(p)^a,\quad f(p^{b+1})=(b+1)f(p)h_f(p)^b,
\end{eqnarray}
and, further,
\begin{eqnarray}
\label{fpeqs2}
f(p^{a+1})^b=(a+1)^bf(p)^bh_f(p)^{ab},\,\,
f(p^{b+1})^a=(b+1)^af(p)^ah_f(p)^{ba}.
\end{eqnarray}
Assume now that $p\in U_f$. Then the right-hand sides of the equations in~(\ref{fpeqs1})
are nonzero and
$f(p^{a+1}),f(p^{b+1})\ne 0$. Therefore, by~(\ref{fpeqs2}),
$$
\frac{f(p^{a+1})^b}{f(p^{b+1})^a}=\frac{(a+1)^bf(p)^b}{(b+1)^af(p)^a}
$$
or, equivalently,
$$
\Big(\frac{f(p^{a+1})}{(a+1)f(p)}\Big)^b=\Big(\frac{f(p^{b+1})}{(b+1)f(p)}\Big)^a.
$$

Second, assume that $U_f$ has at least two elements. If $p,q\in U_f$, then (\ref{leibniz}) and~(\ref{fpa}) imply that
\begin{eqnarray*}
f(p^aq^b)=f(p^a)h_f(q^b)+f(q^b)h_f(p^a)=f(p^a)h_f(q)^b+f(q^b)h_f(p)^a=
\\
\frac{f(p^a)f(q^{b+1})}{(b+1)f(q)}+\frac{f(q^b)f(p^{a+1})}{(a+1)f(p)}.
\end{eqnarray*}

Third, assume additionally that $V_f\ne\emptyset$. Let $p\in V_f$ and $q_1,q_2\in U_f$.
By~(\ref{hfpq}) and the fact that $h_f$ is nonzero-valued,
$$
\frac{f(pq_1)}{f(q_1)}=\frac{f(pq_2)}{f(q_2)}\ne 0.
$$
In other words, we can ``cancel''~$p$ in
$$
\frac{f(pq_1)}{f(pq_2)}=\frac{f(q_1)}{f(q_2)}\ne 0.
$$

Fourth, both the nonzero-valuedness of~$h_f$ and~(\ref{hf}) imply that
$$
f(p^2)\ne 0\quad{\rm for\,all}\,\,p\in U_f.
$$

We have thus found necessary conditions for L-additivity.

\begin{theorem}
\label{neccnds}
Let $f\ne\theta$ be L-additive and $a,b\in\mathbb{N}$. 
\begin{enumerate}
\item[{\rm (i)}] If $p\in U_f$, then
$$
\Big(\frac{f(p^{a+1})}{(a+1)f(p)}\Big)^b=\Big(\frac{f(p^{b+1})}{(b+1)f(p)}\Big)^a.
$$
\item[{\rm (ii)}] If $p,q\in U_f$, then
$$
f(p^aq^b)=\frac{f(p^a)f(q^{b+1})}{(b+1)f(q)}+\frac{f(q^b)f(p^{a+1})}{(a+1)f(p)}.
$$
\item[{\rm (iii)}] If $p\in V_f$ and $q_1,q_2\in U_f$, then
$$
\frac{f(pq_1)}{f(pq_2)}=\frac{f(q_1)}{f(q_2)}\ne 0.
$$
\item[{\rm (iv)}] If $p\in U_f$, then
$$
f(p^2)\ne 0.
$$
\end{enumerate}
\end{theorem}
The question about the sufficiency of these conditions remains open.

To find sufficient conditions for L-additivity, we  study under which conditions we can
apply the procedure described in the proof of Theorem~\ref{uniqhf}
to a given arithmetic function $f\ne\theta$. The function~$h$,
defined as $h_f$ in~(\ref{hf}), must be ($\alpha$)~well-defined, ($\beta$)~c-multiplicative, and
($\gamma$)~nonzero-valued. Condition~($\alpha$) follows from~(iii),
($\beta$) is obvious, and ($\gamma$) follows from (iii) and~(iv).
If the function $g=f/h$ is also c-additive, then $f$ is L-additive
by Theorem~\ref{fgh}. So, we have found sufficient conditions for L-additivity,
and they are obviously also necessary.

\begin{theorem}
\label{necsuffcnds}
An arithmetic function $f\ne\theta$ is L-additive if and only if {\rm (iii)}
and~{\rm (iv)}
in Theorem~{\rm \ref{neccnds}} are satisfied and the function $f/h$ is
c-additive, where
$$
h(n)=\Big(\prod_{p\in U_f}\Big(\frac{f(p^2)}{2f(p)}\Big)^{\nu_p(n)}\Big)
\Big(\prod_{p\in V_f}\Big(\frac{f(pq)}{f(q)}\Big)^{\nu_p(n)}\Big),
\quad q\in U_f.
$$
\end{theorem}

\section{Bounds for an L-additive function}
\label{Bounds}

Let us express~(\ref{n}) as
\begin{eqnarray}
\label{nn}
n=q_1\cdots q_r,
\end{eqnarray}
where $q_1,\dots,q_r\in\mathbb{P}$, $q_1\le\dots\le q_r$.
We first recall the well-known bounds for~$D(n)$ using $n$ and $r$ only.

\begin{theorem}
\label{dbndsthm}
Let $n$ be as in~{\rm (\ref{nn})}. Then
\begin{eqnarray}
\label{dbnds}
rn^\frac{r-1}{r}\le D(n)\le\frac{rn}{2}\le\frac{n\log_2{n}}{2}.
\end{eqnarray}
Equality is attained in the upper bounds if and only if $n$ is a power of~$2$,
and in the lower bound if and only if $n$ is a prime or a power of~$2$.
\end{theorem}

\begin{proof}
See \cite[pp. 118--119]{Ba}, \cite[Theorem~9]{UA}.
\end{proof}

The first upper bound can be improved using the same information.
Westrick~\cite[Ineq.~(6)]{We} presented in her thesis the following bound
without proof.

\begin{theorem}
\label{westrthm}
Let $n$ be as in~{\rm (\ref{nn})}. Then
\begin{eqnarray}
\label{westr}
D(n)\le\frac{r-1}{2}n+2^{r-1}.
\end{eqnarray}
Equality is attained if and only if $n\in\mathbb{P}$ or $q_1=\dots=q_{r-1}=2$.
\end{theorem}
\begin{proof}
If $r=1$ (i.e., $n\in\mathbb{P}$), then (\ref{westr}) clearly holds with equality.
So, assume that $r>1$.

{\em Case 1}. $q_1=\dots=q_{r-1}=2$. Then
$$
D(n)=n\Big(\frac{r-1}{2}+\frac{1}{q_r}\Big)=
\frac{r-1}{2}n+\frac{n}{n/2^{r-1}}={\rm rhs(\ref{westr})},
$$
where ``rhs'' is short for ``the right-hand side''.

{\em Case 2}. $q_1=\dots=q_{r-2}=2$ (omit this if $r=2$) and $q_{r-1}>2$. Since
$$
\frac{1}{q_{r-1}}+\frac{1}{q_r}=
\frac{1}{2}+\frac{4-(q_{r-1}-2)(q_r-2)}{2q_{r-1}q_r}<
\frac{1}{2}+\frac{2}{q_{r-1}q_r},
$$
we have
$$
D(n)<n\Big(\frac{r-2}{2}+\frac{1}{2}+\frac{2}{q_{r-1}q_r}\Big)=
\frac{r-1}{2}n+\frac{2n}{n/2^{r-2}}={\rm rhs(\ref{westr})}.
$$

{\em Case 3}. $q_{r-2}>2$. Then $r\ge 3$ and
$$
D(n)\le n\Big(\frac{r-3}{2}+\frac{1}{3}+\frac{1}{3}+\frac{1}{3}\Big)=
\frac{r-1}{2}n<{\rm rhs(\ref{westr})}.
$$
The claim with equality conditions is thus verified. Because
$$
\frac{rn}{2}-\Big(\frac{r-1}{2}n+2^{r-1}\Big)=\frac{n}{2}-2^{r-1}\ge
\frac{2^r}{2}-2^{r-1}=0,
$$
the upper bound (\ref{westr}) indeed improves~(\ref{dbnds}).
\end{proof}

We extend the upper bounds~(\ref{dbnds}) and~(\ref{westr}) under the assumption
\begin{eqnarray}
\label{hassump}
h_f(p)\ge p\quad{\rm for\,\,all\,\,}p\in U_f.
\end{eqnarray}
Let $n$ in~(\ref{nn}) have $q_{i_1},\dots,q_{i_s}\in U_f$. We
denote
\begin{eqnarray}
\label{s}
p_1=q_{i_1},\,\dots,\,p_s=q_{i_s}
\end{eqnarray}
and
\begin{eqnarray}
\label{Mm}
M=\max_{1\le i\le r}f(q_i)=\max_{1\le i\le s}f(p_i).
\end{eqnarray}

\begin{theorem}
Let $f\ne\theta$ be nonnegative and L-additive satisfying~{\rm (\ref{hassump})}.
Then
\begin{eqnarray}
\label{extndbarb}
f(n)\le\frac{sM}{2}h_f(n)\le\frac{M\log_2n}{2}h_f(n),
\end{eqnarray}
where $s$ is as in~{\rm (\ref{s})} and $M$ is as in~{\rm (\ref{Mm})}. Equality is attained if
and only if $n$ is a power of~$2$.
\end{theorem}

\begin{proof}
By~(\ref{f}) and simple manipulation,
\begin{eqnarray*}
f(n)=h_f(n)\sum_{i=1}^r\frac{f(q_i)}{h_f(q_i)}=
h_f(n)\sum_{i=1}^s\frac{f(p_i)}{h_f(p_i)}\le
h_f(n)M\sum_{i=1}^s\frac{1}{p_i}\le
\\
h_f(n)M\sum_{i=1}^s\frac{1}{2}=
h_f(n)M\frac{s}{2}\le
h_f(n)M\frac{r}{2}\le
h_f(n)M\frac{\log_2n}{2}.
\end{eqnarray*}
The equality condition is obvious.
\end{proof}

\begin{theorem}
Let $f\ne\theta$ be nonnegative and L-additive satisfying~{\rm (\ref{hassump})}.
Then
\begin{eqnarray}
\label{extndwestr}
f(n)\le\Big(\frac{s-1}{2}h_f(n)+h_f(2^{s-1})\Big)M,
\end{eqnarray}
where $s$ is as in~{\rm (\ref{s})} and $M$ is as in~{\rm (\ref{Mm})}.
Equality is attained if and only if $n\in\mathbb{P}$ or
$p_1=\dots=p_{s-1}=2=h_f(2)$.
\end{theorem}

\begin{proof}
If $s=1$ (i.e., $n\in\mathbb{P}$), then (\ref{extndwestr}) clearly holds with equality.
So, assume that $s>1$.

{\em Case 1}. $p_1=\dots=p_{s-1}=2$. Then
\begin{eqnarray*}
f(n)=f(2^{s-1}p_s)=f(2^{s-1})h_f(p_s)+f(p_s)h_f(2^{s-1})=\qquad\qquad
\\
(s-1)f(2)h_f(2^{s-2})h_f(p_s)+f(p_s)h_f(2^{s-1})\le\qquad\qquad
\\
\big((s-1)(h_f(2^{s-2})h_f(p_s)+h_f(2^{s-1})\big)M\le\qquad\qquad
\\
\Big((s-1)h_f(2^{s-2})h_f(p_s)\frac{h_f(2)}{2}+h_f(2^{s-1})\Big)M=
\Big(\frac{s-1}{2}h_f(n)+h_f(2^{s-1})\Big)M.
\end{eqnarray*}

{\em Case 2}. $p_1=\dots=p_{s-2}=2$ (omit this if $s=2$) and $p_{s-1}>2$.
If $s\ge 3$, then
\begin{eqnarray*}
f(n)=f(2^{s-2}p_{s-1}p_s)=f(2^{s-2})h_f(p_{s-1}p_s)+f(p_{s-1}p_s)h_f(2^{s-2})=
\qquad\qquad
\\
(s-2)f(2)h_f(2^{s-3})h_f(p_{s-1}p_s)+f(p_{s-1}p_s)h_f(2^{s-2})=
\\
\frac{s-2}{2}f(2)h_f(2^{s-2})h_f(p_{s-1}p_s)+
\big(f(p_{s-1})h_f(p_s)+f(p_s)h_f(p_{s-1})\big)h_f(2^{s-2})\le
\\
\Big(\frac{s-2}{2}h_f(2^{s-2})h_f(p_{s-1}p_s)+(h_f(p_{s-1})+
h_f(p_s))h_f(2^{s-2})\Big)M=
\\
\Big(\frac{s-2}{2}h_f(n)+(h_f(p_{s-1})+h_f(p_s))h_f(2^{s-2})\Big)M=
\qquad\qquad\qquad\qquad\qquad\qquad
\\
\Big(\frac{s-1}{2}h_f(n)+(h_f(p_{s-1})+h_f(p_s))h_f(2^{s-2})-\frac{1}{2}h_f(n)\Big)M.
\end{eqnarray*}
The last expression is obviously an upper bound for~$f(n)$ also if $s=2$.
If
$$
(h_f(p_{s-1})+h_f(p_s))h_f(2^{s-2})-\frac{1}{2}h_f(n)\le h_f(2^{s-1}),
$$
i.e.,
$$
2(h_f(p_{s-1})+h_f(p_s))-h_f(p_{s-1})h_f(p_s)\le 2h_f(2),
$$
then (\ref{extndwestr}) follows. Since
\begin{eqnarray*}
h_f(p_{s-1})h_f(p_s)-2(h_f(p_{s-1}+h_f(p_s))+4=
(h_f(p_{s-1})-2)(h_f(p_s)-2)\ge
\\
(p_{s-1}-2)(p_s-2)>0,
\end{eqnarray*}
we actually have a stronger inequality
$$
2(h_f(p_{s-1})+h_f(p_s))-h_f(p_{s-1})h_f(p_s)<4.
$$

{\em Case 3}. $p_{s-2}>2$. Then $s\ge 3$ and
\begin{eqnarray*}
f(n)=f(p_1)h_f(p_2\cdots p_s)+f(p_2\cdots p_s)h_f(p_1)=
f(p_1)\frac{h_f(n)}{h_f(p_1)}+f(p_2\cdots p_s)h_f(p_1)
\\
\le\frac{Mh_f(n)}{2}+f(p_2\cdots p_s)h_f(p_1).
\end{eqnarray*}
Since
\begin{eqnarray*}
f(p_2\cdots p_s)h_f(p_1)=\big(f(p_2)h_f(p_3\cdots p_s)+f(p_3\cdots p_s)h_f(p_2)\big)h_f(p_1)=
\\
f(p_2)\frac{h_f(n)}{h_f(p_2)}+f(p_3\cdots p_s)h_f(p_1p_2)\le
\frac{Mh_f(n)}{2}+f(p_3\cdots p_s)h_f(p_1p_2),
\end{eqnarray*}
we also have
$$
f(n)\le 2\frac{Mh_f(n)}{2}+f(p_3\cdots p_s)h_f(p_1p_2).
$$
Similarly,
\begin{eqnarray}
\label{ineq}
f(n)\le\frac{s-3}{2}Mh_f(n)+f(p_{s-2}p_{s-1}p_s)h_f(p_1\cdots p_{s-3}).
\end{eqnarray}
Because
\begin{eqnarray*}
f(p_{s-2}p_{s-1}p_s)=f(p_{s-2})h_f(p_{s-1}p_s)+f(p_{s-1})h_f(p_{s-2}p_s)+
f(p_s)h_f(p_{s-2}p_{s-1})
\\
\le Mh_f(p_{s-2}p_{s-1}p_s)\Big(\frac{1}{p_{s-2}}+\frac{1}{p_{s-1}}+\frac{1}{p_s}\Big)\le
\qquad\qquad\qquad\qquad\qquad\qquad\qquad
\\
Mh_f(p_{s-2}p_{s-1}p_s)\Big(\frac{1}{3}+\frac{1}{3}+\frac{1}{3}\Big)=Mh_f(p_{s-2}p_{s-1}p_s),
\end{eqnarray*}
it follows from~(\ref{ineq}) that
$$
f(n)\le\frac{s-3}{2}Mh_f(n)+Mh_f(n)=\frac{s-1}{2}Mh_f(n).
$$
In other words, (\ref{extndwestr}) holds strictly.

The proof is complete. It also includes the equality conditions.
\end{proof}

If we do not know~$s$ (but know~$r$), we can substitute $s=r$ in~(\ref{extndbarb})
and~(\ref{extndwestr}). We complete this section by extending the lower bound~(\ref{dbnds}).

\begin{theorem}
Let $f$ be nonnegative and L-additive, and let $n$ be as in~{\rm (\ref{nn})} with
$$
h_f(q_1),\dots,h_f(q_r)>0.
$$
Then
$$
f(n)\ge rmh_f(n)^\frac{r-1}{r},
$$
where
$$
m=\min_{1\le i\le r}f(q_i).
$$
Equality is attained if and only if $n$ is a prime or a power of~$2$.
\end{theorem}

\begin{proof}
By~(\ref{f}) and the arithmetic-geometric
mean inequality,
\begin{eqnarray*}
f(n)=h_f(n)\sum_{i=1}^r\frac{f(q_i)}{h_f(q_i)}\ge
h_f(n)m\sum_{i=1}^r\frac{1}{h_f(q_i)}\ge
h_f(n)m\frac{r}{(h_f(q_1)\cdots h_f(q_r))^\frac{1}{r}}=
\\
h_f(n)m\frac{r}{h_f(q_1\cdots q_r)^\frac{1}{r}}=
h_f(n)m\frac{r}{h_f(n)^\frac{1}{r}}=
rh_f(n)^{1-\frac{1}{r}}m.
\end{eqnarray*}
The equality condition is obvious.
\end{proof}

\section{Concluding remarks}
\label{Concluding}
In order to extend the concepts of arithmetic derivative and arithmetic partial
derivative, we first defined the concept of arithmetic subderivative.
As a further extension, we defined the concept of L-additive function. For simplicity,
we stated (contrary to~\cite{HMT2}) that $h_f$ must be nonzero-valued. If we allow $h_f$ to be zero, it turns out that
we then just meet extra work without gaining in results.

Which properties of the arithmetic derivative can be extended to
arithmetic subderivatives and, further, to L-additive functions?
As we saw above, this question can be answered, at least, within certain bounds for the arithmetic
derivative.

\end{document}